\documentclass{amsart}%[11pt,reqno,a4paper]
\usepackage[english]{babel}
\usepackage[centertags]{amsmath}
\usepackage{amsfonts,amssymb,amsthm}
\usepackage{hyperref}  					 %%%% buono per PDF	
\usepackage{mathrsfs}
\usepackage[latin1]{inputenc}
\usepackage{color}
\usepackage[sans]{dsfont}

\usepackage{graphicx}

\numberwithin{equation}{section}

\theoremstyle{plain}
\newtheorem{theorem}{Theorem}[section]

\newtheorem{prop}[theorem]{Proposition}

\newtheorem{cor}[theorem]{Corollary}

\newtheorem{definition}[theorem]{Definition}

\theoremstyle{definition}
\newtheorem{remarkbase}[theorem]{Remark}
\newenvironment{remark}{\pushQED{\qed} \remarkbase}{\popQED\endremarkbase}

\newcommand{\1}{\mathcal{X}}

\newcommand{\ve}{\varepsilon}

\newcommand{\N}{{\mathbb N}}
\newcommand{\R}{{\mathbb R}}

\renewcommand{\d}{\delta}

\newcommand{\vf}{\varphi}

\newcommand{\dist}{ {\mathrm{dist}} }

\def\ba{\begin{aligned}}
\def\ea{\end{aligned}}
\def\beginm{\begin{multline}}
\def\endm{\end{multline}}

\title{Nonnegative multiplicative controllability\\ for semilinear multidimensional reaction-diffusion equations%\\ in heigh dimension
}

\author[G. Floridia]{Giuseppe Floridia} 
\address{Giuseppe Floridia\newline
Universit\`a Mediterranea di Reggio Calabria,\newline
 Dipartimento Patrimonio, Architettura, Urbanistica,\newline
Via dell'Universit\`a 25,
 Reggio Calabria, 89124, Italy} 
\email{floridia.giuseppe@icloud.com}

\begin{document}
\markboth{G. Floridia}{Nonnegative multiplicative controllability for semilinear multidimensional reaction-diffusion equations %in heigh dimension
}
\maketitle
\begin{abstract}
\noindent In this paper we consider a multidimensional semilinear reaction-diffusion equation 
and we obtain at any arbitrary time an approximate controllability result between nonnegative states using as control term the reaction coefficient, that is via multiplicative controls.
\end{abstract}
\begin{small}

\noindent
\emph{Keywords}:{\;Multidimensional semilinear reaction-diffusion equations; approximate controllability; multiplicative controls; nonnegative states. 
%controllabi\-lity; 
}

\noindent
\emph{AMS Subject Classification}{: \, 93C20, 35K57, 35K58,  35K10}.
\end{small}

\date{} 
\pagestyle{plain}

\section{Introduction} 
%\subsection{Problem formulation}\label{Pbfor}
%\noindent 
Let $n\in\N,$ $\Omega%\subset\R^n
$ a bounded open subset of $\R^n$, $T>0,$ $Q_T := \Omega \times (0, T),$ and let denote by $(x,t)$ the generic element of the Cartesian product $Q_T.$ 
Let consider 
%let us introduce t
the following semilinear parabolic Cauchy-Dirichlet boundary value problem 
\begin{equation}\label{PS}
   \begin{cases}
\quad   u_t \; = \; \Delta u%u_{xx} \;
 + \; v(x,t)  u \; + \; f(u)
&\quad  {\rm in} \;\;\;Q_T = \Omega \times (0, T)\,, %\;\; 
%\Omega = (0, 1), 
%\;\; T>0,
\\
\quad u\mid_{\partial\Omega} = 0,
&\quad\qquad\;\;\; t \in (0, T),
\\
\quad u\mid_{t = 0} \; = u_0\in L^2( \Omega), % \in H^1_0 ( 0,1 ).
\end{cases}
\end{equation}
where %$(x,t)\in Q_T,$  
$v \in L^\infty (Q_T),$ and the nonlinear term $f:\R\rightarrow\R$ is supposed to be a Lipschitz function, that is, there exists 
%with $f(0)=0,$ differentiable at $0$ 
a positive constant $L$ such that %and $L$ will denote a Lipschitz constant for $f,$ that is,
\begin{equation}\label{1.2a}
\mid f(u_1) - f (u_2) \mid \; \leq \; L \mid u_1 - u_2 \mid ,\;\;\;\; \forall u_1,u_2 \in \R,
\end{equation}
moreover, we assume that
\begin{equation}\label{zero}
f(0)=0.
\end{equation}

In this paper we study the global approximate controllability properties of the semilinear problem \eqref{PS},
where the control function, that is the variable coefficient through which we can act on the process, is the reaction coefficient $v(x,t),$ that in literature is called \textit{multiplicative control} (see, e.g., \cite{KB}, \cite{CFK}, \cite{F1}, and \cite{FNT}).

Let us recall briefly the classical {\it well-posedness} of the system \eqref{PS}. So, we need to consider
%\subsubsection*{Functional setting and Well-posedness}
the standard  
Sobolev spaces:   
\begin{align*}
H^1 (\Omega) &= \{\phi\in L^2 (\Omega) \mid \; \phi_{x} \in L^2 (\Omega)
\}\\
H_0^1 (\Omega) &= \{\phi\in H^1 (\Omega)  \mid 
 \phi|_{\partial\Omega}=0 
\}\\
H^2 (\Omega) &= \{\phi\in H^1 (\Omega)  \mid \; 
\phi_{x_ix_i}\in L^2 (\Omega), \,i=1,\ldots,n 
\}.
\end{align*}
By classical well-posedness results (see, for instance, Theorem 6.1 in \cite{Lad}, pp. 466-467) problem \eqref{PS} 
with initial data $u_0\in L^2(\Omega)$ admits a unique solution 
 $$u\in L^2(0,T; H^1_0 (\Omega) )\cap  C ( [0, T]; L^2 (\Omega)).$$ 
Furthermore, if $u_0\in H^1_0(\Omega)$, then the solution $u$ of problem (\ref{PS}) satisfies
$$u\in H^1(0,T; L^2 (\Omega) )\cap  C ( [0, T];  H^1_0(\Omega))\cap  L^2 ( 0, T;  H^2(\Omega)).$$%}
The above functional spaces are equipped with the standard norms. Moreover, in this paper we will use $\|\cdot\|,$ $\|\cdot\|_\infty$ and $\langle\cdot,\cdot\rangle,$ %$\Omega$
 instead of
%the bounded open interval (-1,1)
the norms $\|\cdot\|_{L^2(\Omega)}$ and $\|\cdot\|_{L^\infty(Q_T)},$ and the inner product $\langle\cdot,\cdot\rangle_{L^2(\Omega)}$, respectively. 
%Moreover, we will sometimes use $\|\cdot\|_\infty$ instead of $\|\cdot\|_{L^\infty(Q_T)}.$  

Now we can present the main result in Theorem \ref{T1}, where we prove that
the system \eqref{PS}  is {\it nonnegatively globally approximately controllable in $L^2(\Omega)$ at any 
time $T>0,$ by means of % piecewise static
multiplicative controls $v.$} 
We will see that the multiplicative controls $v$ have a simple structure, that is $v$ are piecewise static functions, in the sense of the following definition. 
\begin{definition}\label{piece}
We say that a %bilinear control
function $v\in L^\infty(Q_T)$ is {\it piecewise static} (or a {\it simple function} with respect to the variable {\it t}), if there exist $m\in\N,$ $v_k(x)\in L^\infty(Q_T)$ and $t_k\in [0,T], \,t_{k-1}<t_k,\, k=1,\dots,m$ with $t_0=0 \mbox{ and } t_m=T,$ 
such that 
\begin{equation}\label{alpha_piece}
v(x,t)=v_1(x)\1_{[t_{0},t_1]}(t)+\sum_{k=2}^m v_k(x){\1}_{(t_{k-1},t_k]}(t),
\end{equation}
 where ${\1}_{[t_{0},t_1]}\,  \mbox{  and  }  \,{\1}_{(t_{k-1},t_k]}$ are the indicator function of $[t_{0},t_1]$ and $(t_{k-1},t_k]$, respectively.
Sometime, for clarity purposes, we will call the function $v$ in \eqref{alpha_piece} a {m-steps} piecewise static function
\end{definition}

%\begin{definition}
%%We say that 
%The system \eqref{PS}  is said to be nonnegatively globally approximately controllable in $L^2(-1,1)$ at any %arbitrary small 
%time $T>0,$ by means of %piecewise static
%multiplicative controls $\alpha,$ 
%% function $\alpha\in L^\infty(Q_T)$ is piecewise static, if $\alpha(\cdot,x)$ is piecewise constant in $t$ and $\alpha(t,\cdot)\in L^\infty(-1,1),\; t\in (0,T).$
%if for any nonnegative 
%$u_0,u^*\in L^2(-1,1)$ with $u_0\neq0,$ for every $\varepsilon>0$ %u_d\geq 0$ and any $%\forall
% %u_0\in L^2(-1,1)$ such that
%there exists %\exists\,
%%$ T^*=T^*(\varepsilon,u_0,u^*)> 0$
%%and 
%a piecewise static multiplicative control %\exists
%$\alpha=
%\alpha (\varepsilon,u_0,u^*),\,\alpha
%\in L^\infty(Q_T),$ such that for the correspon\-ding strong solution %$^a$ %(\footnote{See Definition \ref{strong}, for the precise definition of strong solutions.})
% $u(x,t)$ of $(\ref{PS})$
%we obtain
%$$\|u(\cdot,T)-u^*\|_{L^2(-1,1)}< \varepsilon. %,\quad\,\;\;\;\forall T\in(0,T^*].
%$$
%\end{definition}

Finally, we can give the main result.
\begin{theorem}\label{T1}
For any nonnegative 
$u_0,u^*\in L^2(\Omega)$ with $u_0\neq0_{L^2(\Omega)},$ for every $\varepsilon>0,$ and any $T>0$  %u_d\geq 0$ and any $%\forall
 %u_0\in L^2(-1,1)$ such that
there exists %\exists\,
%$ T^*=T^*(\varepsilon,u_0,u^*)> 0$
%and 
a %2-steps 
piecewise static multiplicative control %\exists
$v=
v (\varepsilon, T, u_0,u^*),\,
v\in L^\infty(Q_T),$ such that for the correspon\-ding solution %$^a$ %(\footnote{See Definition \ref{strong}, for the precise definition of strong solutions.})
 $u(x,t)$ of $(\ref{PS})$
we obtain
$$\|u(\cdot,T)-u^*\|_{L^2(\Omega)}< \varepsilon. %,\quad\,\;\;\;\forall T\in(0,T^*].
$$
%The nonlinear degenerate system \eqref{PS}
% is nonnegatively globally approximately controllable in $L^2(-1,1)$ at any 
% time $T>0,$ by means of 2-steps piecewise static multiplicative controls.
\end{theorem}

Theorem \ref{T1} is proved in Section \ref{Proof_main}.\\

%\noindent
 It is useful the following remark.
\begin{remark}
We note that as a consequence of the assumptions \eqref{1.2a} and \eqref{zero} on the nonlinear function $f$ the following inequality holds
\begin{equation}\label{1.2b}
| f(u)| \leq L | u | ,\;\;\;\; \forall u \in \R,
\end{equation}
where $L$  is the Lipschitz constant in (\ref{1.2a}).
\end{remark}
%\begin{remark}
%\label{r1.2}\rm
We observe that the nonnegative control result, given in Theorem \ref{T1},  is consistent with the constraints given by the PDE in \eqref{PS}. Indeed, it also holds 
%on the nonlinear term $f$ the further assumption:
% \begin{equation*}\label{sem_strong}
%% f(0)=0\;\; \text{ and }\;\; 
%{ f(u)\, \text{\it is differentiable at } 0},
% \end{equation*}
%so it follows 
$$\frac{f(u)}{u} \in L^\infty (Q_T),$$ that follows from \eqref{1.2b}, thus
we can extend the strong maximum principle from linear parabolic PDEs (see, e.g, Chapter 2 in \cite{Fr}, p. 34) to semilinear parabolic problem \eqref{PS}, since the terms $v(x,t)u+f(u (x,t))$ can be written as $ \widetilde{v}(x,t) u (x,t) $,  where $  \widetilde{v} := v+\dfrac{f(u)}{u} \in L^\infty (Q_T).$\\
So, the strong maximum principle implies that the system \eqref{PS} cannot be steered anywhere from $u_0 \equiv 0$, and 
if $u_0 (x) \geq 0$ in $\Omega$,  then the corresponding solution to \eqref{PS} remains nonnegative at any moment of time, regardless of the possible choice of the multiplicative control $v$. 
This means the constraint that the system \eqref{PS} cannot be steered from any such nonnegative $u_0\in L^2(\Omega)$ to any target state $u^*\in L^2(\Omega)$ which is negative on a nonzero measure set in the space domain. \\

 To prove Theorem \ref{T1} we need an intermediate and crucial controllability result given in Theorem \ref{lem preserving}, obtained under further regularity assumptions and constraints on the initial and target states.
\begin{theorem}\label{lem preserving} 
Let $u_0,\,%\overline{u}
u^*\in C^2 
%H^1_0
(\Omega)$ %\,,\:u_0,\,\overline{u}\geq0$ in $\Omega.$  
such that $u_0(x)\neq0\: \text{ for every } x\in\Omega,$ and
\begin{equation}\label{ass}
\exists \nu>0:\;\nu\leq\frac{%\overline{u}
u^*(x)}{u_{0}(x)}\leq 1 \quad  %\text{ for }
\forall \,x\in %\left(0,1\right)
\Omega\,.
%\backslash \bigcup_{l=1}^n\partial D_l.
%\left\{x_l\right\}. %\neq 0,1, x_l,\, l=1,\ldots,n.
\end{equation}
Then, for every $\ve>0$ and any $T>0$ %there exists a %sufficiently small time 
%$T=T(\ve,u_{0}, u^*)>0 %\in\left(0,\frac{1}{2L}\right)%\,T=T(\ve,u_{0}, u^*)
%$ and a %for every $T %\in(0,\frac{1}{2L})
there exists a static multiplicative control $v=v(\ve, T, u_{0}, \,u^*)%\overline{u})%,\,v
\in L^\infty(Q_T%(0,1)\times(0,T)
),\, v=v(x),$
 such that%\in \big(0,\frac{1}{4L}\big),$ 
\begin{equation}\label{ineq pres}
\|u (\cdot, T) - %\overline{u}
u^*(\cdot)\|_{L^2(\Omega)}\leq \ve\,,%+\sqrt{2}\|r_{in}\|_{L^2(\Omega)},
\end{equation}%$$
 where $u$ is the corresponding solution to \eqref{PS} on $Q_T.$ Moreover, the static multiplicative control $v$ is the following function   
 $$ v(x,t)=\frac{v_0^*(x)}{T}\quad \forall (x,t)\in Q_T,$$ with 
 $\displaystyle v^*_0 (x):=\ln\left( \frac{%\overline{u} 
u^*(x)}{u_{0} (x) } \right),$ for every $x\in\Omega.$
\end{theorem}
\begin{remark}
Theorem \ref{lem preserving} includes the case $u_0$ and $u^*$ both strictly negative on $\Omega,$ under the condition \eqref{ass}. So, for this result it is just important that both initial and target state have the same sign.
\end{remark}

%\end{remark}
\subsection*{Structure of the paper.} %and some notations}
%\noindent 
We prove the main result, that is Theorem \ref{T1}, %is given 
in Section \ref{Proof_main}. The  proof  of Theorem \ref{T1} need the intermediate and crucial result given in Theorem \ref{lem preserving}, that is proved in Section \ref{intermediate} together with some useful PDE estimates.
%In the final Section \ref{Appl} we motivate our study about multiplicative controllability for reaction-diffusion equations presenting some related applications.

\subsection*{State of the art %in multiplicative controllability and degenerate pa\-rabolic equations %contents and structure%: Why multiplicative controllability?
}\label{art}
The nonnegative approximate controllability results for reaction-diffusion equations are consistent with the strong maximum principle constraints.
In literature, for this kind of results we have to refer to the  pioneering papers by  A.Y. Khapalov, contained in the book \cite{KB}, where it is obtained for reaction diffusion equations the nonnegative approximate controllability  via multiplicative controls, before in large time, and after in small time under very strong assumptions on the initial and target states. In this paper, we are able to remove those strong constraints on the data in our general Theorem \ref{T1}.
Moreover, the proof of that theorem %we introduce a new %kind of 
%proof, %(this proof is easy extend to the above cases of \cite{CFproceedings1}, \cite{CF2} and \cite{F1}), 
%that 
permits us to obtain the nonnegative controllability first in
 {\it arbitrary small time}, then at any time by an iteration argument. %instead of {\it large time}. 
%Moreover, the proof, contained in %\cite{CFproceedings1}, \cite{CF2} and 
% \cite{F1}, of the nonnegative controllability in large time for the $(SDeg)$ case has the further obstruction that permitted to treat only superlinear growth, with respect to $u,$ of the nonlinearity function $f(x,t,u).$ While the new proof, adopted in this paper, %\cite{F3}, %(see also Section 5 of \cite{CFK}) 
%permits us to control also linear growth of $f,$ with respect to $u.$ 
 This proof is inspired by the recent paper \cite{F3} of the author, where a similar result is proved in the unidimensional setting for degenerate reaction-diffusion equations. About nonnegative controllability results for degenerate parabolic equations we also mention \cite{F1}. Moreover,  in \cite{V} Vancostenoble proved a nonnegative controllability result in large time for a linear parabolic equation with singular potential, following the approach of \cite{CFproceedings1} and \cite{CF2}.
 %The new ideas introduced in \cite{F3} have been a good preliminary to approach the main goals of this paper, that is, the multiplicative controllability between sign-changing states. 
For completeness we need recall some recent results about the approximate multiplicative controllability for unidimensional reaction-diffusion equations between sign-changing states, see \cite{CFK},  %in suitable classes of functions that change sign: 
 %in \cite{CFK} 
 by the author with Cannarsa and Khapalov 
 regarding a semilinear uniformly parabolic system, 
%\eqref{1.1}
%uniformly parabolic equations 
and \cite{FNT}, by the author with Nitsch and Trombetti, about degenerate parabolic equations. \\
Recently, it is also studied the exact controllability for evolution equations via bilinear controls.  See, e.g., \cite{ACU} and \cite{ACU2} by Alabau-Boussouira, Cannarsa and Urbani, \cite{CU1} by Cannarsa and Urbani, and \cite{DL} by Duprez and Lissy.\\
  %in \cite{FNT} by the author with Nitsch and Trombetti.  %\textcolor{blue}{
Finally,  an interesting work in progress, related to this paper is the problem of the approximate controllability via multiplicative control for nonlocal operators, applied to the fractional heat equation studied in \cite{BWZ} by Biccari, Warma and Zuazua. 
%Other interesting open problems are suggested by the papers \cite{FR} and \cite{KCPF}.
 %\newpage

\section{Some PDE estimates and an intermediate goal}\label{intermediate}
We prove the crucial intermediate goal given in Theorem \ref{lem preserving} in Section \ref{second_result}. For the proof of Theorem \ref{lem preserving} and for that of the main result, that is Theorem \ref{T1}, we need before some general PDE estimates for the solution of the problem \eqref{PS}, that we give before in Section \ref{estimates}.
\subsection{Some PDE estimates}\label{estimates}
Let us start this section by the statement of Proposition \ref{lem PDE} that we prove immediately below. 
\begin{prop}\label{lem PDE} 
Let $T\in(0,\frac{1}{4L}].$ %\in(0,\frac{1}{4L}).$ 
Let $u_0\,%\overline{u}
%u^*
\in %C^1 
H^1_0
(\Omega),$ $v\in C^2(Q_T)$ 
with $v(x,t)\leq0$ on $Q_T,$ and let
%for a.e. $(x,t)
%\in Q_T.$ 
 $u$ %$u\in H^1(0,T; L^2 (\Omega) )\cap  C ( [0, T];  H^1_0(\Omega))\cap  L^2 ( 0, T;  H^2(\Omega))$
 the corresponding unique solution to \eqref{PS}. Then, we have
 $$f(u)\in C ([0,T]; L^2 (\Omega))$$ and the following estimates hold:
  \begin{itemize}
 \item[$\star$] 
 $\quad\;\| u \|_{C ([0,T]; L^2 (\Omega))} \; \leq \; \sqrt{2} \|  u_{0} \|_{L^2 (\Omega)},
$
\item[$\star$] $\: \| f(u) \|_{C ([0,T]; L^2 (\Omega))} 
%\; \leq \; %\sqrt{2}
% L  \| u \|_{C ([0,T]; L^2 (\Omega))}
 \leq\sqrt{2} L \|  u_{0} \|_{L^2 (\Omega)},
$
\item[$\star$] $\;\;\:\,\quad\quad\ \| \Delta u \|_{L^2 (Q_T)}\leq C(L,T,v) \|  u_{0} \|_{H^1_0 (\Omega)},$
\end{itemize}
where $ L $  is the Lipschitz constant in (\ref{1.2a}) and 
$$C(L,T, v):=\sqrt{1 +
2T\max_{x \in \overline{\Omega}} |\Delta v|
+  2 L^2 T}\;.$$
\end{prop}

\begin{proof}
We start this proof by evaluating $ \| u \|_{C ([0,T]; L^2 (\Omega))}.$
Since $ v (x,t) \leq 0$  for a.e. $(x,t)
\in Q_T,$ multiplying by $u$ the equation in (\ref{PS}), integrating by parts and using \eqref{1.2b}
 yields 
\begin{align*}
\frac{1}{2}\int_0^t \int_\Omega (u^2)_t dx ds =&\int_0^t \int_\Omega u\,u_t  dx ds \\
=& \;
\int_0^t \int_\Omega u\,\Delta u\, dx ds+
%\frac{1}{T} 
\int_0^t \int_\Omega v\, u^2 dx ds+\int_0^t  \int_\Omega   f(u)u \,dx ds
\\
\leq &
-\int_0^t \int_\Omega 	|\nabla u|^2\, dx ds+
L\int_0^T  \int_\Omega   u^2 dx dt
\leq L
\int_0^T  \int_\Omega   u^2 dx dt,
\end{align*}
where $ L $  is the Lipschitz constant in (\ref{1.2a}). Then, since $T\in \big(0,\frac{1}{4L}\big)$ we deduce
\begin{align*}\label{4.8}
\!\!\!\!\!\!\int_\Omega u^2 (x,t) dx &\leq \int_\Omega u_{0}^2 (x) dx + %(L^2+1)
2L\int_0^T  \int_\Omega   u^2 dx dt\\
\!&\nonumber\leq \int_\Omega u_{0}^2 (x) dx  + 2L%(L^2+1)
T \|u\|^2_{C([0,T];L^2(\Omega))}\\
\nonumber&\leq  \|  u_{0} \|^2_{L^2 (\Omega)}+ \frac{1}{2} \|u\|^2_{C([0,T];L^2(\Omega))},\;\;t\in(0,T),
\end{align*}
thus,
\begin{equation}\label{4.10}
\| u \|_{C ([0,T]; L^2 (\Omega))} \; \leq \; \sqrt{2} \|  u_{0} \|_{L^2 (\Omega)},
\end{equation}
that proves the first estimate in the statement of Proposition \ref{lem PDE}. 

Now, we evaluate $ \| f(u) \|_{C ([0,T]; L^2 (\Omega))}$. 
From %assumptions 
\eqref{1.2b}  and (\ref{4.10})
 it easy follows that $$f(u)\in C ([0,T]; L^2 (\Omega)),$$ and the following estimate holds
\begin{equation}\label{4.11}
 \| f(u) \|_{C ([0,T]; L^2 (\Omega))} \; \leq \; %\sqrt{2}
 L  \| u \|_{C ([0,T]; L^2 (\Omega))}\leq\sqrt{2} L \|  u_{0} \|_{L^2 (\Omega)},
\end{equation}
that proves the second estimate in Proposition \ref{lem PDE}.

Finally, we evaluate $ \| \Delta u \|_{L^2 (Q_T)}$. 
%{\bf Step 2:} {\it Evaluation of $  \parallel  \Delta w  + f (u) \parallel^2_{L^2 (Q_T)}\!\!.$}  %$ \parallel  w_{xx} \parallel_{L^2 (Q_T)}$
% this assumption will be removed in Step. 3.\\
Multiplying by $\Delta u$ the equation in \eqref{PS}, 
%with the previous choice of 
integrating 
over $ Q_T,$ and applying Young's inequality we obtain
\begin{align*}
\parallel  \Delta u \parallel^2_{L^2 (Q_T)}&=\!\!\int_0^T \int_\Omega u_t \Delta u dx dt -%\frac{1}{T}
 \int_0^T \int_\Omega %v^*_{0}
v u \Delta u dx dt-\int_0^T  \int_\Omega   f(u)\Delta u\, dx dt\\%\nonumber\\
&\leq \int_0^T \int_\Omega u_t \Delta u dx dt \,-
%\frac{1}{T} 
\int_0^T \int_\Omega v %v^*_{0} 
u \Delta u dx dt\\ 
&\qquad\qquad\qquad\qquad\quad+ \: 
\frac{1}{2}  \int_0^T  \int_\Omega   f^2(u) dx dt%\nonumber\\&\qquad\qquad\qquad\qquad\qquad\qquad\qquad
+ \frac{1}{2}  \int_0^T  \int_\Omega   |\Delta u|^2
%w^2_{xx} 
dx dt\, .
\end{align*}
Thus, integrating by parts, keeping in mind the boundary condition in \eqref{PS}, taking into account that the reaction term $v$ is such that $v(x,t)% = \frac{1}{T} v_{0}(x) 
\leq 0$ for every $(x,t)\in Q_T,$ and using the estimates \eqref{4.11} and \eqref{4.10}  %that from the assumption \eqref{ass} 
it follows
%In this step, let us suppose that 
%$v^*_0\in C^2 (\Omega)$ and $v^*_0 (x) \leq 0,$ %since the , 
%Thus, integrating by parts and recalling that $ v_0 (x) \leq 0 $, 
that
\begin{align*}%\label{4.7}
\|\Delta u \|^2_{L^2 (Q_T)}&\leq 2\int_0^T \int_\Omega u_t \Delta u dx dt \; - \;
%\frac{2}{T} 
2\int_0^T  \int_\Omega v\,
%v^*_{0}
 u\, \Delta u dx dt 
\; + \:  \int_0^T  \int_\Omega   f^2(u) dx dt\\
&\leq \; - \int_0^T \int_\Omega (	|\nabla u|^2%_{x}
)_t dx dt \; + \;
%\frac{2}{T} 
2\int_0^T  \int_\Omega %v^*_{0}
v |\nabla u|^2 dx dt\\
&\qquad\qquad\qquad\quad+
%\frac{1}{T}
 \int_0^T \int_\Omega \nabla v%v^*_{0}
\cdot \nabla(u^2)dx dt
\; + \: T\, \| f(u) \|^2_{C ([0,T]; L^2 (\Omega))}\\%  \int_0^T  \int_\Omega   f^2(u) dx dt	\\
&\leq \; \int_\Omega  |\nabla u_0|^2
%_{0\,x} 
 dx \;% - \int_\Omega u^2_{x}(T,x) dx dt  \;
 - \;
%\frac{1}{T}
 \int_0^T  \int_\Omega \Delta %v^*_{0} 
v\,u^2 dx dt
\; + 2 L^2 T\|  u_{0} \|^2_{L^2 (\Omega)}%\:   \int_0^T  \int_\Omega   f^2(w) dx dt
\\
&\leq
% \int_\Omega u^2_{in\,x} 
% dx 
  \int_\Omega  |\nabla u_0|^2
%_{0\,x} 
 dx 
 +
%\frac{1}{T} 
\max_{x \in \overline{\Omega}} | \Delta v
 %v_{0}^*
  |
\int_0^T \int_\Omega   u^2 dx dt 
+  2 L^2 T\|  u_{0} \|^2_{L^2 (\Omega)}\\
&\leq
\|  \nabla u_{0} \|^2_{L^2 (\Omega)}
% \int_\Omega  |\nabla u_0|^2
% dx 
 +
\max_{x \in \overline{\Omega}} |\Delta v
  |
\int_0^T \| u \|^2_{C ([0,T]; L^2 (\Omega))} dt 
+  2 L^2 T\|  u_{0} \|^2_{L^2 (\Omega)}\\
&\leq
%\| \nabla u_{0} \|^2_{L^2 (\Omega)}
% \int_\Omega  |\nabla u_0|^2
% dx 
\left(1 +
2T\max_{x \in \overline{\Omega}} |\Delta v|
+  2 L^2 T\right)\|  u_{0} \|^2_{H^1_0 (\Omega)}.
\end{align*}
\end{proof}
In the following Proposition \ref{lem PDE2}, we generalize the first estimate in Proposition \ref{lem PDE} to the case of a general reaction coefficient $v\in L^\infty(Q_T).$
\begin{prop}\label{lem PDE2} 
Let $T>0.$ %\in(0,\frac{1}{4L}).$ 
Let $u_0\in %C^2 
H^1_0
(\Omega),\,v\in L^\infty(Q_T),$ and let $u$ the corresponding unique solution to \eqref{PS}. Then, we have
% $$f(u)\in C ([0,T]; L^2 (\Omega))$$ and the following estimates hold:
%  \begin{itemize}
% \item[$\star$] 
 $$\| u \|_{C ([0,T]; L^2 (\Omega))} \; \leq \; e^{(L+\|v^+\|_\infty)T} \|  u_{0} \|_{L^2 (\Omega)},$$
%\item[$\star$] $\: \| f(u) \|_{C ([0,T]; L^2 (\Omega))} 
%%\; \leq \; %\sqrt{2}
%% L  \| u \|_{C ([0,T]; L^2 (\Omega))}
% \leq\sqrt{2} L \|  u_{0} \|_{L^2 (\Omega)},
%$
%\item[$\star$] $\;\;\:\,\quad\quad\ \| \Delta u \|_{L^2 (Q_T)}\leq C(L,T,v) \|  u_{0} \|_{H^1_0 (\Omega)},$
%\end{itemize}
where $ L $  is the Lipschitz constant in (\ref{1.2a}) and $v^+(x,t)=\max\{v(x,t),0\}$ is the positive part of $v.$
%$$C(L,T, v):=\sqrt{1 +
%2T\max_{x \in \overline{\Omega}} |\Delta v|
%+  2 L^2 T}\;.$$
\end{prop}

\begin{proof}
%Now, we have to evaluate $ \| w\|_{C ([0,t_1]; L^2 (0,1))}.$  
%Multiplying by $ w $ the equation in the first problem of (\ref{w}), integrating by parts, and arguing as in the proof of
%%proceeding similarly to obtain the inequality 
%\eqref{4.8}, for every $t\in (0,t_1),$
%we have
%\begin{align*}
%\frac{1}{2}\int_0^t \int_\Omega (w^2)_t dx ds %=&%\int_0^t \int_\Omega w_t w dx ds \\
%&=
%-\int_0^t \int_\Omega w^2_{x}\, dx ds+m \int_0^t \int_\Omega w^2 dx ds+\int_0^t  \int_\Omega   f(w)w dx ds\\
%&\leq (m+L)%\frac{1}{2}(L^2+1)
%\int_0^t  \int_\Omega   w^2 dx ds,
%\end{align*}
%where $ L $  is as in (\ref{1.2a}).
%Then
%$$
%\int_\Omega w^2(x,t) dx %=&%\int_0^t \int_\Omega w_t w dx ds \\
%\leq \int_\Omega u_{in}^2(x) dx 
%+ 2(m+L)%\frac{1}{2}(L^2+1)
%\int_0^t  \int_\Omega   w^2 dx ds,\;\;t\in (0,t_1).
%$$
%We start this proof by evaluating $ \| u \|_{C ([0,T]; L^2 (\Omega))}.$
%Since $ v (x,t) \leq 0$  for a.e. $(x,t)
%\in Q_T,$ 
%
Proceeding as in the proof of Proposition \ref{lem PDE}, that is multiplying by $u$ the equation in (\ref{PS}), integrating by parts and using \eqref{1.2b}
 we obtain 
\begin{align*}
\frac{1}{2}\int_0^t \int_\Omega (u^2)_t dx ds %=&\int_0^t \int_\Omega u\,u_t  dx ds \\
=& \;
\int_0^t \int_\Omega u\,\Delta u\, dx ds+
%\frac{1}{T} 
\int_0^t \int_\Omega v\, u^2 dx ds+\int_0^t  \int_\Omega   f(u)u \,dx ds
\\
\leq &
-\int_0^t \int_\Omega 	|\nabla u|^2\, dx ds+ \int_0^t \int_\Omega v^+\, u^2 dx ds+
L\int_0^t  \int_\Omega   u^2 dx dt\\
\leq &\; (L+\|v^+\|_\infty)
\int_0^t  \int_\Omega   u^2 dx ds,\;\;\;\;\forall t\in(0,T),
\end{align*}
then
$$\int_\Omega u^2(x,t)dx \leq \int_\Omega u_0^2(x)dx+(L+\|v^+\|_\infty)
\int_0^t  \int_\Omega   u^2 dx ds,\;\;\;\;\forall\, t\in(0,T).$$
%where $ L $  is the Lipschitz constant in (\ref{1.2a}). 
Thus, applying Gr\"onwall's inequality we deduce %since $T\in \big(0,\frac{1}{4L}\big)$
$$
\|u(\cdot,t)\|^2_{L^2(\Omega)}%=\int_\Omega w^2 (x,t) dx 
\leq e^{2(L+\|v^+\|_\infty)T}\|u_{0}\|^2_{L^2(\Omega)},\;t\in (0,T),
$$
from which it follows the conclusion of the proof.
%\begin{equation}\label{w energy}
%\|u\|_{C([0,T];L^2(\Omega))}%=\int_\Omega w^2 (x,t) dx 
%\leq e^{(L+\|v^+\|_\infty)T}\|u_{0}\|_{L^2(\Omega)}%x=Ke^{L\,t_1}\|u_{in}\|_{L^2(0,1)}
%\; .
%\end{equation}
%
%Then, %since $T\in \big(0,\frac{1}{4L}\big)$ 
%we obtain
%\begin{align*}\label{4.8}
%\!\!\!\!\!\!\int_\Omega u^2 (x,t) dx &\leq \int_\Omega u_{0}^2 (x) dx + %(L^2+1)
%2(L+\|v^+\|_\infty)
%\int_0^T  \int_\Omega   u^2 dx dt.\\
%\!&\nonumber\leq \int_\Omega u_{0}^2 (x) dx  + 2L%(L^2+1)
%T \|u\|^2_{C([0,T];L^2(\Omega))}\\
%\nonumber&\leq  \|  u_{0} \|^2_{L^2 (\Omega)}+ \frac{1}{2} \|u\|^2_{C([0,T];L^2(\Omega))},\;\;t\in(0,T),
%\end{align*}
%thus,
%\begin{equation}\label{4.10}
%\| u \|_{C ([0,T]; L^2 (\Omega))} \; \leq \; \sqrt{2} \|  u_{0} \|_{L^2 (\Omega)},
%\end{equation}
%that proves the first estimate in the statement of Proposition \ref{lem PDE}. 
\end{proof}
From Proposition \ref{lem PDE2} we can easy obtain the following Corollary \ref{cor PDE2}.
\begin{cor}\label{cor PDE2} 
Let $T>0.$ %\in(0,\frac{1}{4L}).$ 
Let $u^0_1,\,u^0_2\in %C^2 
H^1_0
(\Omega),\,v\in L^\infty(Q_T),$ and let $u_1$ and $u_2$ the unique solution of \eqref{PS} corresponding  to $u_1^0$  and $u_2^0,$  respectively. Then, we have
% $$f(u)\in C ([0,T]; L^2 (\Omega))$$ and the following estimates hold:
%  \begin{itemize}
% \item[$\star$] 
 $$\| u_1-u_2 \|_{C ([0,T]; L^2 (\Omega))} \; \leq \; e^{(L+\|v^+\|_\infty)T} \|  u^0_{1}-u^0_{2} \|_{L^2 (\Omega)},$$
%\item[$\star$] $\: \| f(u) \|_{C ([0,T]; L^2 (\Omega))} 
%%\; \leq \; %\sqrt{2}
%% L  \| u \|_{C ([0,T]; L^2 (\Omega))}
% \leq\sqrt{2} L \|  u_{0} \|_{L^2 (\Omega)},
%$
%\item[$\star$] $\;\;\:\,\quad\quad\ \| \Delta u \|_{L^2 (Q_T)}\leq C(L,T,v) \|  u_{0} \|_{H^1_0 (\Omega)},$
%\end{itemize}
where $ L $  is the Lipschitz constant in (\ref{1.2a}) and $v^+(x,t)=\max\{v(x,t),0\}$ is the positive part of $v.$
%$$C(L,T, v):=\sqrt{1 +
%2T\max_{x \in \overline{\Omega}} |\Delta v|
%+  2 L^2 T}\;.$$
\end{cor}
\begin{proof} Let set
$w:=u_1-u_2,$ we note that $w$ satisfy the following Cauchy-Dirichlet     
\begin{equation}\label{PSw}
\begin{cases}
\quad   w_t \; = \; \Delta w%u_{xx} \;
 + \; v(x,t)  w \; + \; f(u_1)-f(u_2)
&\quad  {\rm in} \;\;\;Q_T = \Omega \times (0, T)\,, %\;\; 
%\Omega = (0, 1), 
%\;\; T>0,
\\
\quad w\mid_{\partial\Omega} = 0,
&\quad\qquad\;\;\; t \in (0, T),
\\
\quad w\mid_{t = 0} \; = u_1^0-u_2^0\in H^1_0( \Omega). % \in H^1_0 ( 0,1 ).
\end{cases}
\end{equation}
Following the same idea of the proof of Proposition \ref{lem PDE2}, that is 
multiplying by $w$ the equation in (\ref{PSw}), integrating by parts and using \eqref{1.2a}
 we obtain 
\begin{align*}
\frac{1}{2}\int_\Omega w^2(x,t) dx %=&\int_0^t \int_\Omega u\,u_t  dx ds \\
=&\frac{1}{2}\int_\Omega w^2(x,0) dx+
\int_0^t \int_\Omega w\,\Delta w\, dx ds\\&
%\frac{1}{T} 
\quad\quad\;+\int_0^t \int_\Omega v\, w^2 dx ds
%&\qquad\qquad\quad
+\int_0^t  \int_\Omega  \left(f(u_1)-f(u_2)\right)w \,dx ds
\\
\leq &  \frac{1}{2}\|u_1^0-u_2^0\|^2_{L^2(\Omega)}
%-\int_0^t \int_\Omega 	|\nabla w|^2\, dx ds%+ \left(\right)\int_0^t \int_\Omega v^+\, u^2 dx ds+
%L\int_0^t  \int_\Omega   u^2 dx dt\\
+ (L+\|v^+\|_\infty)
\int_0^t  \int_\Omega   w^2 dx ds,\;\;\;\;\forall t\in(0,T),
\end{align*}
%then
%$$\int_\Omega u^2(x,t)dx \leq \int_\Omega u_0^2(x)dx+(L+\|v^+\|_\infty)
%\int_0^t  \int_\Omega   u^2 dx ds,\;\;\;\;\forall\, t\in(0,T).$$
%where $ L $  is the Lipschitz constant in (\ref{1.2a}). 
Then, applying Gr\"onwall's inequality we obtain %since $T\in \big(0,\frac{1}{4L}\big)$
$$
\|u_1(\cdot,t)-u_2(\cdot,t)\|^2_{L^2(\Omega)}%=\int_\Omega w^2 (x,t) dx 
\leq e^{2(L+\|v^+\|_\infty)T}\|u_1^{0}-u_2^{0}\|^2_{L^2(\Omega)},\;\;\;\,\forall t\in (0,T),
$$
from which the conclusion.
\end{proof}
\subsection{The proof of Theorem \ref{lem preserving} }\label{second_result}

Now, we are ready to give the proof of Theorem \ref{lem preserving} .

\begin{proof}[Proof of Theorem \ref{lem preserving}.]
Let us set $\displaystyle v^*_0 (x):=\ln\left( \frac{%\overline{u} 
u^*(x)}{u_{0} (x) } \right),$ for every $x\in\Omega.$  From \eqref{ass} we note that 
$v^*_0\in L^\infty (\Omega)$  
and for the static multiplicative control $v$ we have
%$$%\Delta w
\begin{equation}\label{sign}
v (x,t):= \frac{v^*_0 (x)}{T}\leq0\;\;\; \text{ for every } (x,t)\in Q_T.
\end{equation}
%$$
%$
%v^*_0 (x) \; \leq 0,%\forall 
%\text{ for every } x\in \Omega.$ %(see \eqref{ass}).
%We select the bilinear control
%$$%\Delta w
%v (x,t) \; := \; \frac{1}{T} v_0 (x).
%$$
Then, let us compute the solution $ u $ of \eqref{PS} at time $T$, corresponding to the previous choice of the reaction coefficient $v,$ using the following representation. %formula
\begin{equation}\label{WT}
u(x,T) \; = \; %\overline{u} 
u^*(x)  \; + \int_0^T e^{v^*_0 (x) \frac{(T - t)}{T}} (\Delta u (x, \tau) + f (u(x,\tau))) dt,\quad \forall x\in\Omega.
\end{equation}
%{\bf Step 1:} {\it Representation formula for $ u (\cdot, T) $.}

The formula \eqref{WT} is obtained in the following way.
For every fixed $\bar{x}\in \Omega,$ let us consider the non-homogeneous first-order ODE
 $$u^\prime(\bar{x},t)=\frac{v^*_0 (\bar{x})}{T}  u(\bar{x},t) \; +\big( \Delta u(\bar{x},t)+f(u(\bar{x},t))\big)\;\;\;\; t\in (0,T),$$
  associated to \eqref{PS}.
Then, we easy obtain that the %corresponding 
solution $u$ 
%to the first problem in \eqref{w} admits
has the following representation formula
$$
u(x,t) \; = \; e^{v^*_0 (x) \frac{t}{T}} u_{0} (x)  \; + \; %\mathop{
\int_0^t e^{v^*_0 (x) \frac{(t - \tau)}{T}} (\Delta u(x, \tau) + f (u(x, \tau)))d \tau,\,\;\;\;\forall (x,t)\in Q_T,
$$
so for $ t = T$ we obtain \eqref{WT}.
%\begin{equation}\label{WT}
%u(x,T) \; = \; %\overline{u} 
%u^*(x)  \; + \int_0^T e^{v^*_0 (x) \frac{(T - t)}{T}} (\Delta u (x, \tau) + f (u(x,\tau))) dt,\quad \forall x\in\Omega.
%\end{equation}

From \eqref{WT}, using \eqref{sign} and H\"older's inequality, and applying the estimates in Proposition \ref{lem PDE}, %is nonpositive,
we deduce the following inequalities %estimate
\begin{align}\label{uTf}
\|u(x,T)-%\overline{u}
u^* (x)&\|^2_{\! L^2(\Omega)}\!\!=\!\!\!
\int_\Omega \left(\int_0^T e^{v^*_{0} (x)
\frac{(T - \tau)}{T}} (\Delta u(x, \tau) + f (u(x,\tau))) d \tau \!\!\right)^2\!\!\!\!dx
\\
\nonumber&\leq T \|\Delta u  + f (u)\|^2_{L^2 (Q_T)}\\\nonumber &\leq 2\,T\left(\| \Delta u\|_{L^2 (Q_T)}^2+\|f (u)\|_{L^2 (Q_T)}^2
\right)\\
\nonumber &\leq 2\,T\left[\left(1 +
2T\max_{x \in \overline{\Omega}} |\Delta v|
+  2 L^2 T\right)+2 L^2T
\right]\|u_0\|^2_{H^1_0(\Omega)}.\\
\nonumber &= 2\,T\left(1 +
2\max_{x \in \overline{\Omega}} |\Delta v_0^*|
+  4 L^2 T
\right)\|u_0\|^2_{H^1_0(\Omega)}\,.
\end{align}
We note that in the last equality we have done the replacement $v=\frac{v_0^*}{T}$ in the constant $C(L,T,v)$ of Proposition \ref{lem PDE} and for the sequel it is crucial that constant (and consequently the last member of \eqref{uTf}) it is still bounded, as $ T \rightarrow 0^+$, also with that choice of the \lq\lq {\it singular}'' (in {\it T}) reaction coefficient $v$.  

Finally, fixed $\ve>0,$ since the last side of (\ref{uTf}) goes to zero as $ T \rightarrow 0^+$, there exists $T_0^*\in(0,\frac{1}{4L}),\, T_0^*=T_0^*(\ve,v_0^*)$ such that for every $T\in (0,T_0^*]$ we obtain
$$\|u(\cdot,T)-%\overline{u}
u^* \|^2_{\! L^2(\Omega)}\!\!< \ve,$$
that is the approximate controllability at any time $T\in(0,T^*]$.\\
%$ u(\cdot, T) \rightarrow \overline{u} $ in $ L^2 (\Omega)$ at the same time.
%
%from which %the conclusion %of the proof 
%follows the approximate controllability at any $T\in(0,T^*]$, since $T_1>0$ was arbitrarily small.
%Moreover, if $T>T^*$ using the Step 1-2 we obtain the approximate controllability at time $T^*,$ after we restart close to $u^*$ and we stabilize the system into the neighborhood of $u^*,$ applying the Step 1-2 $n$ times, for some $n\in\N$ $T>T^*,$ from a suitable approximation of $u^*$ to it.\\
%
Furthermore, if $T>T^*_0$ we can prove the approximate controllability at time $T,$ using an argumentation introduced by the author in \cite{F3} (see the proof of Theorem 1.4 in that paper) for unidimensional degenerate reaction-diffusion equations.
%Step 1-2 
So, we first obtain the ap\-proximate controllability at time $T^*_0.$ Then, we restart at time $T^*_0$ from a state close to $u^*,$ and we stabilize the system into the neighborhood of $u^*,$ applying the above strategy 
%Step 1-2 
overall $n$ times, for some $n\in\N,$ on $n$ small time intervals by measure $\frac{T-T^*_0}{n},$ steering the system in any interval from a suitable approximation of $u^*$ to $u^*$.
\end{proof}

\section{The proof of the main result}\label{Proof_main}
%\label{Proof_main}
%\subsection{Nonnegative controllability}\label{control section}
%\begin{equation}\label{u pres proof}
%\begin{cases}
%\quad   u_t \; = \; \left(a(x)u_{x}\right)_x \; + \; \alpha(x,t)  u+ f(x,t,u)
%&  {\rm in} \;\;\;Q_{T}= (-1, 1) \times (0, T),\\
%\quad B.C. %\quad u (0,t) = u (1,t) = 0,
%&\quad t \in (0, T),
%\\
%\quad u(x,0)%\mid_{t = 0}
% \; = u_{0}(x),\;\;\;& \quad x\in(-1,1),%+ r_{in},
%\end{cases}
%\end{equation}
%where %$(0, T)$ is a generic time interval, 
% $u_{0}%,\,r_{in}
% \in 
%H^1_a(-1,1),\,u_0,u^*\geq0, u_0\not\equiv0.$ %and $u_{0}$ has $n$ points of sign change at $ x_{l} \in (-1,1),\,l=1,\ldots,n, 
%
%In this section, 
%\noindent 
Let us give the proof of the main result of this paper, Theorem \ref{T1}.
%Now we need to extend the above result to the general case.
%that is, we have to prove 
%Theorem \ref{T1}.
\begin{proof}{(Proof of Theorem \ref{T1}).}
\noindent Let fix $\ve>0$. % a fixed real number. 
%For any $\displaystyle \ve\in\left[0,\frac{1}{2}\right),\,$ 
%%with $\displaystyle \rho_0:= %\frac{1}{2}
%%\!\!\min_{l=0,\ldots,n} \big\{ x_{l+1}-x_{l}\big\},$ 
%let us consider
%the  set 
%$\displaystyle \Sigma_{\ve}:=(-1+\ve, 1-\ve) 
%%\bigcup_{l=0}^{n} (x_l + \rho, x_{l+1} -\rho)
%.$ %\bigcup \ldots \bigcup (x_n+\sigma, 1-\sigma)
%By approximating 
Since $u_{0},\,u^*\in L^2(\Omega)$ by approximating  there exist
  %Without loss of generality we can suppose 
%$\overline{\rho}\in(0,\frac{\rho_0}{2})$ and 
%u^\ve_{in},
$u^\ve_{0},\,u^*_\ve \in C^2(\overline{\Omega})$ %with the same points of sign change of $u_{0},\overline{u},$ 
%approximating functions, %in $L^2(-1,1)$
% chosen 
 such that: 
%\end{equation}%$
%we have 
 \begin{itemize}
% \item[$\star$]  $%{\overline{u}_\ve}|_{\overline{\Sigma}_{\overline{\rho}}}=u_{0},\;\;
% {u^\ve_0}(x),{u^*_\ve}(x)>0\;\;\; \forall x\in[-1,1];
% $
%\begin{romanlist}[(ii)]
%%\begin{itemlist}
% \item 
 %\begin{equation}\label{5.1}
\item[$\star$] $%\!\!\!\!\!\!\!\displaystyle 
u^\ve_0,\;u^*_\ve>0\, \text{ on }\,%\overline{
\Omega,\;\;%}
$ %\end{equation}
and the quotient function
$%\Big\{
\dfrac{u^*_\ve }{u_{0}^\ve}$
is bounded on 
$\Omega,%\Big\}
$
that is, 
\begin{equation}\label{Sve}
%\text{ there exists }
\exists\; M_\ve:=M(\ve, u_{0},u^*)>0\;:\;  
%\displaystyle M_\ve%=\max\left\{1, \max_{x\in [0,1]}\psi(x)\right\}\geq%
%:=
%%\max_{x\in\overline{\Sigma}_{\overline{\ve}}}
%%\Big\{\frac{\overline{u} (x)}{u_{0} (x)}\Big\}+1> 
%\max%\sup
%_{x\in \overline{\Omega}
%% \Sigma_0
%%(-1,1)\backslash \displaystyle\bigcup_{l=1}^n\left\{x_l\right\}
%%\overline{\Omega}_{\overline{\rho}}
%}
%\Big\{\frac{u^*_\ve (x)}{u_{0}^\ve (x)}\Big\}+1\,.
0< \frac{u^*_\ve (x)}{u_{0}^\ve (x)}\leq M_\ve,\qquad \forall x\in\Omega,
\end{equation}
for the following is useful to choose directly an upper bound $M_\ve>1;$ 
\item[$\star$] $u^*_\ve$ and $u^\ve_0$ satisfy the following approximation conditions
 \begin{equation}\label{5.1}
\| u^*_\ve-u^*
 \|%_{L^2 (-1,1)}
 <\frac{\ve}{4} \;\;\text{  and  }
 \;\; \| u_{0}^\ve-u_{0} \|%_{L^2 (-1,1)}
<%\frac{\sqrt{2}}{16S_\ve e^{\nu}}
 \frac{\ve}{16 e^{L}M_\ve }\,,
 %c_0(M_\ve)
% \ve\,. %\frac{\ve}{4S_\ve e^{2\nu}}\:,
\end{equation}
%is  
%Thus, for the above function there exists the maximum  
where $L$ is the Lipschitz constant in $\eqref{1.2a}.$% and
% let us set
\end{itemize}

%Let us introduce the universal constant $$\displaystyle K=K(u_{in},\overline{u}):=%\max\left\{1, 
%\max_{x\in \overline{\Omega}}\psi(x)+1>1%\right\}
%.$$
%There exists 
For every $x\in\Omega,$ let us define $\displaystyle\dist(x,\partial\Omega):=\inf_{y\in\partial\Omega}|x-y|\,(>0).$ 
Let $\displaystyle %\overline{\eta}
\eta>0$ %(of course $\eta$ is small: $\eta\in(0, \frac{\diam\Omega}{2})$)
such that %we can define % <\frac{\rho_0}{2}:= \frac{1}{2}\min_{l=0,\ldots,n} \big\{ x_{l+1}-x_{l}\big\},$ consider
the following set 
$$\displaystyle \Omega_{\eta}:=%(\rho, x_1-\rho) 
\left\{x\in\Omega: \dist(x,\partial\Omega)>\eta\right\}%\bigcup_{l=0}^{n} \Big(x_l + \rho, x_{l+1} -\rho\Big).%\bigcup \ldots \bigcup (x_n+\sigma, 1-\sigma)
$$
is a not empty open set, then we have $ \overline{\Omega}_{\eta}\subset\Omega.$\\
From \eqref{Sve} we deduce that 
$$%\begin{equation}%\label{5.1d}
%\displaystyle 
0<\min_{x\in\overline{\Omega}_{\eta}}
\Big\{\frac{u^*_\ve(x)}{u^\ve_{0} (x)}\Big\}%\nu(\eta)%=\max\left\{1, \max_{x\in \overline{\Omega}}\psi(x)\right\}\geq%
\leq%\max_{x\in\overline{\Omega}_{\eta}}
%\Big\{
\frac{u^*_\ve(x)}{u^\ve_{0} (x)}\leq M_\ve, \;\;\quad \forall x\in\overline{\Omega}_\eta,%\Big\},
%%\;\;\;%\textit{ for any }
%%\;\forall \rho\in \left(0,\frac{\rho_0}{2}\right).
$$%\end{equation}
then, there exists $\nu(\eta)>0$ $\left(\displaystyle\nu(\eta):=\frac{1}{M_\ve}\min_{x\in\overline{\Omega}_{\eta}}
\Big\{\frac{u^*_\ve(x)}{u^\ve_{0} (x)}\Big\}\right)$ such that %$K$ it follows that 
\begin{equation}\label{ass2}
%\displaystyle 
\nu(\eta)%=\max\left\{1, \max_{x\in \overline{\Omega}}\psi(x)\right\}\geq%
\leq%\max_{x\in\overline{\Omega}_{\eta}}
%\Big\{
\frac{u^*_\ve(x)}{M_\ve\,u^\ve_{0} (x)}\leq 1,\;\;\quad \forall x\in\overline{\Omega}_\eta\,.%\Big\},
%%\;\;\;%\textit{ for any }
%%\;\forall \rho\in \left(0,\frac{\rho_0}{2}\right).
\end{equation} 

Let note that \eqref{ass2} has the same structure of the assumption \eqref{ass} in Theorem \ref{lem preserving}, so it is natural to proceed by the following strategy 
%of the proof 
consisting in two control actions: in the first step we drive the system from the initial state $u_0$ to the intermediate target state $M_\ve u^\ve_0,$ then in the second step we steer the system from this intermediate state to $u^*,$ using the crucial Theorem \ref{lem preserving}.

For that we consider a further approximation of $u^*,$ indeed there exist\\ $%\overline{\rho}>0,\,
\eta=\eta(\ve)>0$ and $u^*_{\eta}, v^\eta_0 \in C^2(\overline{\Omega})$ with $v^\eta_0 \leq0$ such that:
\begin{enumerate}
\item[$(i)$]
$
u_\eta^*(x)\; = \;  \left\{ \begin{array}{ll}
u^*_\ve (x), \;\;\;\; &  x \in \Omega_\eta,\\ %(\rho, x_1-\rho) \bigcup (x_1 + \rho, x_2 -\rho) \ldots \bigcup (x_n+\rho, 1-\rho),   \\ 
0, \;\;\;\; &  %{\rm elsewhere}
x \in \overline{(\Omega\backslash\Omega_{\frac{\eta}{2}})} % \; {\rm in} \; \overline{\Omega}_\rho
 ,\\
\end{array}
\right.   
$  \;\;\;\;\;
and
\begin{equation}\label{approx u}
\displaystyle %\lim_{T \rightarrow t_1^+} 
\| u^*_{\eta} -u^*_\ve \|_{L^2 (\Omega)}<\frac{\ve}{4},
\end{equation}
\item[$(ii)$]
$v^\eta_0 (x) = \left\{ \begin{array}{ll}
\ln \left( 
%\frac{\overline{u} (x)}{K u_{in} (x) }
\frac{u^*_\ve(x)}{M_\ve\,u^\ve_{0} (x)} \right), \;\;\;\; &  x \in \Omega_{\eta},\\
%(\rho, x_1-\sigma) \bigcup (x_1 + \sigma, x_2 -\sigma) \ldots \bigcup (x_n+\sigma, 1-\sigma),   \\ 
0, \;\;\;\; &  x \in \overline{(\Omega\backslash\Omega_{\frac{\eta}{2}})}\,, \\
\end{array}
\right.   
$
\end{enumerate}
where $\displaystyle \Omega_{\frac{\eta}{2}}:=%(\rho, x_1-\rho) 
\left\{x\in\Omega: \dist(x,\partial\Omega)>\frac{\eta}{2}\right\}. %\bigcup_{l=0}^{n} \Big(x_l + \rho, x_{l+1} -\rho\Big).%\bigcup \ldots \bigcup (x_n+\sigma, 1-\sigma)
$

Now, keeping in mind that $M_\ve>1,$ we can select the positive constant reaction coefficient %bilinear control %the bilinear control  
\begin{equation}\label{bilinear}
v_1 := \frac{\log M_\ve}{T_1}>0, \;\;\;\;\; \forall\,(x,t) \in \overline{\Omega}\times(0, T_1), \;\;\;\text{ for some } T_1>0.
\end{equation} 
%Moreover, let us set 
%\begin{equation}\label{static2}
%{v_\ve}^*_0 (x) \; = \;  \left\{ \begin{array}{ll}
%\ln \left( 
%%\frac{\overline{u} (x)}{K u_{in} (x) }
%\frac{u^*_\ve(x)}{M_\ve\,u^\ve_{0} (x)} \right), \;\;\;\; &  x \in \Omega_{\eta},\\
%%(\rho, x_1-\sigma) \bigcup (x_1 + \sigma, x_2 -\sigma) \ldots \bigcup (x_n+\sigma, 1-\sigma),   \\ 
%0, \;\;\;\; &  {\rm elsewhere} \; {\rm in} \; \overline{\Omega}\,. \\
%\end{array}
%\right.   
%\end{equation}
Then, let us choose the  2-steps piecewise static (see Definition \ref{piece}) multiplicative control
\begin{equation}\label{piecew}
v(x,t) \; = \;  \left\{ \begin{array}{ll}
\;\;\;v_1\,, \;\;\;\; &  (x,t) \in \overline{\Omega}\times(0, T_1),\\
\\
%(\rho, x_1-\sigma) \bigcup (x_1 + \sigma, x_2 -\sigma) \ldots \bigcup (x_n+\sigma, 1-\sigma),   \\ 
\dfrac{v^\eta_0(x)}{T-T_1}\,, \;\;\;\; &  (x,t) \in \overline{\Omega}\times(T_1, T)\,, \\
\end{array}
\right.   
\end{equation}
where $T_1$ and $T$ will be determined below.\\
Let $u$ be the solution to \eqref{PS} corresponding to the above choice of the multiplicative control $v$ and to the initial state $u_0.$

%to solve our problem we have to impose that $u$ satisfy the following constraints:
%$$%\begin{equation}\label{piecew}
%u(x,t) \; = \;  
%\left\{ \begin{array}{ll}
%\;\;\;M_\ve u^\ve_0\,, \;\;\;\; &  \text{ if }\;\;\;  t=T_1 %(x,t) \in \overline{\Omega}\times(0, T_1)
%,\\
%\\
%%(\rho, x_1-\sigma) \bigcup (x_1 + \sigma, x_2 -\sigma) \ldots \bigcup (x_n+\sigma, 1-\sigma),   \\ 
%\;\;\dfrac{v^\eta_0(x)}{T-T_1}\,, \;\;\;\; &  (x,t) \in \overline{\Omega}\times(T_1, T)\,, \\
%\end{array}
%\right.   
%$$%\end{equation}
%
%
%Note that here 
%\noindent {\bf Step ?:} 
{\it Steering the system from $u^\ve_{0}%+r_{in}%+\kappa (\cdot, t_1) + h (\cdot, t_1) 
$ to $M_\ve u^\ve_{0},$ at some time $T_1>0.$}
Let us denote by $u^\ve(x,t)$ 
%and $u(x,t)$  %In this step, %we represent 
the solution %$ u $ 
of \eqref{PS} with initial state $u_0^\ve$. % and $u_0$, respectively. %as the sum of two functions $ w (x,t) $ and $h (x,t)$, %which solve the following problems in $Q_T:$
%which 
%solves the problem %in \eqref{w} 
%in $\displaystyle(-1,1)\times(0,t_1),\, t_1>0.$  
%Applying the constant bilinear control  $ \alpha(x,t) = \alpha_1 > 0, \forall x\in (-1,1),$ on  the interval $(0, t_1),$ 
Thus, 
%keeping in mind the abstract formulation \eqref{operatorNL} for the problem \eqref{PS}, the Duhamel's principle and the Proposition \eqref{loclip}, 
the solution $u^\ve(x,t)$, at some time $T_1,$ is represented in Fourier series in the following way
%\textcolor{blue}{
%\begin{multline}\label{w conv}
%w (x,t_1)
%= \; H (x, t_1) \; + I(x,t_1),
%%+ K u_{0} (x),
%\end{multline}
%where 
\begin{equation}\label{w conv}%\label{5.4}
u^\ve (x,T_1)
%H (x,t_1) 
= \; %S_\ve 
e^{v_1 T_1}\,\sum_{k = 1}^\infty   e^{-\lambda_k T_1} %\left(
\langle u^\ve_0,\vf_k\rangle
%\int_{-1}^1 u^\ve_{in} (r) \omega_p(r)%\sin p\pi r
%  \, dr \right)
   \vf_k(x)
%I (x,T_1) 
%\sin p\pi  x
+ F_\ve(x,T_1)\,
 %\nonumber
\end{equation}
with \;$\displaystyle F_\ve(x,T_1):= \sum_{k = 1}^\infty   \left[\int_0^{T_1} e^{(v_1 -\lambda_k ) (T_1 -t)} 
%\left(
\langle 
f (u^\ve(\cdot, t)),\vf_k\rangle
%\int_{-1}^1 f (r,t,u^\ve(r, t))  \omega_p(r)%\sin p\pi r
% \, dr  
% %\right) 
dt
  \right]\vf_k(x)%\sin p\pi  x \,,
,$\\ where
%\begin{align}\label{w conv}%\label{5.4}
%u^\ve (x,T_1)
%%H (x,t_1) 
%&= \;S_\ve \sum_{p = 1}^\infty   e^{-\lambda_p^2 t_1} %\left(
%\langle u^\ve_0,\omega_p\rangle
%%\int_{-1}^1 u^\ve_{in} (r) \omega_p(r)%\sin p\pi r
%%  \, dr \right)
%   \omega_p(x)\\
%%I (x,t_1) 
%& %\sin p\pi  x
%+ \sum_{p = 1}^\infty   \left[\int_0^{t_1} e^{(\alpha_1 -\lambda_p^2 ) (t_1 -t)} 
%%\left(
%\langle 
%f (\cdot,t,u^\ve(\cdot, t)),\omega_p\rangle
%%\int_{-1}^1 f (r,t,u^\ve(r, t))  \omega_p(r)%\sin p\pi r
%% \, dr  
%% %\right) 
%dt
%  \right]\omega_p(x)%\sin p\pi  x \,,
% \nonumber
%\end{align}
$\{-\lambda_k\}_{k\in\N}$
%$\{\lambda_k\}_{k\in\N}$ is an increasing sequence with
%$\lambda_k\longrightarrow +\infty\, \mbox{ as } \, k \, \rightarrow\infty\,,$
%such that the 
are the eigenvalues of the Laplacian operator $A_0 u:=\Delta u$ (we note that $\lambda_k\geq0$ and $\lambda_k\leq\lambda_{k+1},$ for every $k\in\N,$ and $\lambda_k\rightarrow+\infty,$ as $k\rightarrow+\infty$),
% $(A_0,D(A_0)),$ defined in \eqref{DA0},
%(\footnote{The domain $D(A_0)$ of the operator $A_0$ coincides with the domain $D(A)$ of the operator $(A,D(A)),$ defined in \eqref{DA}.}) 
%are given by $\{-\lambda_p\}_{p\in\N}$, 
and $\{\varphi_k\}_{k\in\N}$ are the corresponding eigenfunctions %$\{\omega_p\}_{p\in\N},$
 that form a complete orthonormal system in $L^2(\Omega).$ %tsee Proposition \ref{spectrum}. %\:\footnote{
 We trivially remark that  the eigenvalues of the operator 
 %$(A,D(A)),$ with 
 $Au:=\Delta u+v_1u$ are obtained by a shift, that is we have $\{-\lambda_k+v_1\}_{k\in\N},$ and the corresponding eigenfunctions are the same %as $(A_0,D(A_0))$, that is 
 $\{\varphi_k\}_{p\in\N}$.\\%}. %$\{\omega_p\}_{p\in\N},$
 By the strong continuity semigroup property of the heat equation, %see Proposition \ref{str cont}, %defined in \eqref{} 
we deduce %that
$$
 \sum_{k = 1}^\infty   e^{-\lambda_k T_1} 
% \left(\int_0^1 u^\ve_{in} (r) \omega_p(x)%\sin p\pi r
%  \, dr \right)
  \langle u^\ve_0,\vf_k\rangle \vf_k(x)\longrightarrow u^\ve_0%(\cdot)
   \;\;\;\; \text{ in } L^2(-1,1) \;\;  \text{ as }\;\; T_1\rightarrow0.
 $$
 So, there exists a small time $T'_1\in(0,1),\; T'_1=T'_1(\ve,u_0,u^*),
%\leq\left\{\frac{}{},1\right\}
$ such that, keeping also in mind that $e^{v_1T_1}=M_\ve$ by \eqref{bilinear}, we deduce
 \begin{equation}\label{R1}
%\left
\| %e^{\alpha_1 T_1}
u^\ve(\cdot,T_1)
%M_\ve\sum_{p = 1}^\infty   e^{-\lambda_p T_1} %\left(
%\langle u^\ve_0,\omega_p\rangle\omega_p(\cdot)
-M_\ve u^\ve_0(\cdot)%\right
\|<\frac{\ve %\overline{T}_1
}{32}+\|F_\ve(\cdot,T_1)\|,\qquad \forall\: T_1\in(0,T'_1%\overline{T}_1
].
\end{equation}
%{\it Bound for $\| H (\cdot, T_1)\|_{L^2 (-1,1)}.$}
%In the same
%way, %applying also the 
%Since %$u^\ve$  is a strict solution, by Proposition \ref{f in L2} we have 
%$f(\cdot,\cdot,u^\ve(\cdot,\cdot))\in L^2(Q_T),$ 
%using
%%} 
%%assumption \eqref{fsigni}, 
%Proposition \refl{loclip} %{f in L2} %obtained in the proof of previous Lemma,
Using H\"older's inequality, Parseval's identity, the inequality \eqref{1.2b}, and Proposition \ref{lem PDE2} we deduce
\begin{align}\label{H1}
\|&F_\ve(x,T_1)\|^2%_{L^2 (-1,1)}
=
 \sum_{k = 1}^\infty \left| \int_0^{T_1} e^{(v_1 -\lambda_k) (T_1 -t)}
 \langle 
f (u^\ve(\cdot, t)),\vf_k\rangle
% \left( \int_{-1}^1  %(e^{-(p \pi)^2 T_1} -1) %e^{m T_1} 
%f (r,t,u (r, t))  \omega_p(r) \, 
%dr \right)
 dt
\right|^2%}
%\nonumber
\\
&\leq  \,%\Big
 \sum_{k = 1}^\infty \left( \int_0^{T_1}  e^{2(v_1 -\lambda_k) (T_1 -t)}dt%(e^{-(p \pi)^2 t_1} -1) %e^{m t_1} 
\right)\cdot\left(\int_0^{T_1}\left|
\langle 
f (u^\ve(\cdot, t)),\vf_k\rangle
%\int_{-1}^1f (r,t,u (r, t))  \omega_p(r) \, dr
\right|^2\!dt\right) %\right)
%}
\nonumber\\
 &\leq%\sqrt{
 %\left( \int_0^{t_1}  
 e^{2v_1T_1}T_1
\int_0^{T_1}\sum_{k = 1}^\infty \left|\langle 
f (u^\ve(\cdot, t)),\vf_k\rangle
%\int_{-1}^1f (r,t,u (r, t))  \omega_p(r) \, dr
\right|^2\!dt %\right)
=M_\ve^2 %\sqrt{
T_1 %\left( \int_0^{t_1}  
\int_0^{T_1} \|f (u^\ve(\cdot, t))\|^2%_{L^2(-1,1)}
%\int_{-1}^1f ^2(r,t,u (r, t)) dr
dt %\right)
\nonumber\\
&\leq M_\ve^2 %\sqrt{
T_1 %\left( \int_0^{t_1}  
\int_0^{T_1} \|u^\ve\|^2%_{L^2(-1,1)}
%\int_{-1}^1f ^2(r,t,u (r, t)) dr
dt\leq %\gamma_0^{2}
c(T_1) L M_\ve^2 %e^{2k \vt T_1}
T_1\|u^\ve_{0}\|^{2}_{L^2 (\Omega)
}\,,\nonumber%\nonumber
\end{align}
%where $C=C(\|u^\ve_{0}\|_{1,a%L^2 (-1,1)
%})$ and $k$ are the positive constants introduced in the statement of Proposition \ref{f in L2}.\\
%Now, we have to evaluate $ \| w\|_{C ([0,t_1]; L^2 (-1,1))}.$  
%Multiplying by $ w $ the equation in the first problem of (\ref{w}), integrating by parts, and arguing as in the proof of
%%proceeding similarly to obtain the inequality 
%\eqref{4.8}, for every $t\in (0,t_1),$
%we have
%\begin{align*}
%\frac{1}{2}\int_0^t \int_0^1 (w^2)_t dx ds %=&%\int_0^t \int_0^1 w_t w dx ds \\
%&=
%-\int_0^t \int_0^1 w^2_{x}\, dx ds+\alpha_1 \int_0^t \int_0^1 w^2 dx ds+\int_0^t  \int_0^1   f(w)w dx ds\\
%&\leq (\alpha_1+L)%\frac{1}{2}(L^2+1)
%\int_0^t  \int_{-1}^1   w^2 dx ds,
%\end{align*}
%where $ L $  is as in (\ref{1.2a}).
%Then
%$$
%\int_{-1}^1 w^2(x,t) dx %=&%\int_0^t \int_0^1 w_t w dx ds \\
%\leq \int_{-1}^1 u_{0}^2(x) dx 
%+ 2(\alpha_1+L)%\frac{1}{2}(L^2+1)
%\int_0^t  \int_{-1}^1   w^2 dx ds,\;\;t\in (0,t_1).
%$$
%Thus, by Gronwall's lemma we have %since $T\in \big(0,\frac{1}{4L}\big)$
%$
%\|w(t,\cdot)\|^2_{L^2(0,1)}%=\int_0^1 w^2 (x,t) dx 
%\leq e^{2(\alpha_1+L)t_1}\|u_{0}\|^2_{L^2(0,1)},\;t\in (0,t_1).
%$
%So,
%\begin{equation}\label{w energy}
%\|w\|_{C([0,t_1];L^2(0,1))}%=\int_0^1 w^2 (x,t) dx 
%\leq e^{(\alpha_1+L)t_1}\|u_{0}\|_{L^2(0,1)}=Ke^{L\,t_1}\|u_{0}\|_{L^2(0,1)}\;\; .
%\end{equation}
From \eqref{R1} using \eqref{H1} it follows that there exists $T^*_1\in (0, T'_1],\, T^*_1=T^*_1(\ve,u_0,u^*),$ such that
\begin{equation}\label{small}
\| %e^{\alpha_1 T_1}
u^\ve(\cdot,T_1)
%M_\ve\sum_{p = 1}^\infty   e^{-\lambda_p T_1} %\left(
%\langle u^\ve_0,\omega_p\rangle\omega_p(\cdot)
-M_\ve u^\ve_0(\cdot)%\right
\|<\frac{\ve %\overline{T}_1
}{32}+\|F_\ve(\cdot,T_1)\|\leq\frac{\ve}{16},\;\;\qquad\forall\: T_1\in(0,T^*_1].
\end{equation}
Using Corollary \ref{cor PDE2} and the inequaity \eqref{small}, keeping in mind \eqref{bilinear} and \eqref{5.1}, for every $T_1\in(0,T^*_1],$
%the approximation condition $(i)$ about the initial state $u_0,$ 
we obtain 
\begin{align}\label{5.10}\|u (\cdot, T_1) -& M_\ve u^\ve_0 (\cdot)\|%_{L^2({-1},1)}
\leq \|u (\cdot, T_1) - u^\ve (\cdot, T_1)\|%_{L^2({-1},1)}
+\|u^\ve (\cdot, T_1) - M_\ve u^\ve_0 (\cdot)\|\\%_{L^2({-1},1)}
\nonumber&< e^{(L+\|v_1^+\|_\infty)T_1}\|u_0-u_0^\ve\|+ \frac{\ve}{16}
\leq e^{(L+v_1T_1)}\frac{\ve}{16e^LM_\ve}%\|u_0-u_0^\ve\|
+ \frac{\ve}{16}=\frac{\ve}{8}.%\\
%\nonumber&=e^{L}M_\ve\|u_0-u_0^\ve\|+ \frac{\ve}{4}\leq \frac{\ve}{2}.
\end{align}
Let set 
%\begin{equation}\label{sigma_0e}
$\d_0^\ve(x):=u(x, T_1) - M_\ve u^\ve_0 (x),$ 
%\end{equation}
%we note that 
by \eqref{5.10} we have
\begin{equation}\label{sigma}
\|\delta_0^\ve\|<\frac{\ve}{4\sqrt{2}}\,.
\end{equation}
%{\bf Step 2:} 

{\it Steering the system from $M_\ve u^\ve_{0}+\delta_0^\ve%+\overline{r}_{in}%+\kappa (\cdot, t_1) + h (\cdot, t_1) 
$ (at time $T_1>0)$ to $u^*$ at time $T,$ for $T>T_1.$}
%For every $\rho \in \left(0,\frac{\rho_0}{2}\right),$ let us set
%$$
%\overline{u}_\rho(x)\; = \;  
%\left\{ \begin{array}{ll}
%\overline{u}_\ve (x), \;\;\;\; & {\rm if} \;x \in \Omega_\rho,\\ %(\rho, x_1-\rho) \bigcup (x_1 + \rho, x_2 -\rho) \ldots \bigcup (x_n+\rho, 1-\rho),   \\ 
%0, \;\;\;\; &  {\rm if}  %{\rm elsewhere} \; {\rm in} \; 
%\;x \in [-1, 1]\backslash \Omega_\rho,\\
%\end{array}
%\right.
%$$  
%then, we choose $%\overline{\rho}>0,\,
%\overline{\rho}\in \left(0,\frac{\rho_0}{2}\right)$ such that
%\begin{equation}\label{approx u}
%\displaystyle %\lim_{T \rightarrow t_1^+} 
%\parallel  \overline{u}_{\overline{\rho}} -\overline{u}_\ve \parallel_{L^2 (-1,1)}<\frac{\ve}{4}.
%\end{equation}
In this step let us restart at time $T_1$ from the intermediate state $M_\ve u^\ve_{0}+\d_0^\ve$ and our goal is to steer the system arbitrarily close to $u^*.$ \\
%Let us consider
%\begin{equation}\label{a2}
%\|e^{\alpha_{\ve j}}S_\ve u^\ve_{0}-u_\ve^*\|<\frac{\ve}{12},\qquad\qquad  \forall j\in\N \quad \text{with} \quad\: j\geq j^*\,.
%\end{equation}
%
Let us consider the following semilinear Dirichlet problem
 \begin{equation}\label{PSfinal}
%\label{PS}
\left\{\begin{array}{l}
\displaystyle{u_t-\Delta u=%\alpha(x,t)
%\frac{\alpha_{\ve j} (x)}{T-T_1} 
\dfrac{v^\eta_0(x)}{T-T_1} u+ f(u)\;\;\;\; \mbox{ in } \; \widetilde{Q}_T:=\Omega\times(T_1,T)
}\\ [2.5ex]
%\displaystyle{ 
%}\\ [2.5ex]
%\displaystyle{a(x)u_x(t,x)|_{x=\pm 1} = 0\,\,\qquad\qquad\qquad\qquad\qquad\;\;\,\,\, t\in(0,T) }\\ [2.5ex]
\displaystyle{u\mid_{\partial\Omega} = 0,\qquad\qquad\quad\quad\qquad\quad\quad\;\quad\quad\,\;\;\;\: t\in (T_1,T)%\in L^2(-1,1)\, %\,\qquad\qquad\quad\qquad\qquad\;\; \,x\in(-1,1)
}
%\\ [2.5ex]
%\displaystyle{\alpha(x,t)\in L^\infty(Q_T)\, %\,\qquad\qquad\quad\qquad\qquad\;\; \,x\in(-1,1)
%}
~,
\end{array}\right.
\end{equation}
%\begin{equation}\label{PSfinal}
%   \begin{cases}
%\quad   u_t \; = \; \Delta u%u_{xx} \;
% + \; v(x,t)  u \; + \; f(u)
%&\quad  {\rm in} \;\;\;Q_T = \Omega \times (0, T)\,, %\;\; 
%%\Omega = (0, 1), 
%%\;\; T>0,
%\\
%\quad u\mid_{\partial\Omega} = 0,
%&\quad\qquad\;\;\; t \in (0, T),
%\\
%\quad u\mid_{t = 0} \; = u_0\in L^2( \Omega), % \in H^1_0 ( 0,1 ).
%\end{cases}
%\end{equation}
%%Let us fix an arbitrary $j\in\N$ with $j\geq j^*,$ and let us choose as control
%%the following static multiplicative function
%%\begin{equation}\label{control2}
%%%\alpha_2=
%\alpha (x,t) \; := \; \frac{1}{T-T_1} \alpha_{\ve j} (x)\leq0     \;\;\;\:\; \forall (x,t) \in \widetilde{Q}_T:=(-1,1)\times(T_1,T),
%\end{equation}
%and 
and we denote by $\tilde{u}(x,t)$ the unique solution to
 %that solves the problem 
 \eqref{PSfinal} with 
%\begin{itemize}
%\item time interval $(T_1,T)$ instead of $(0,T);$ 
%\item multiplicative control given by \eqref{control2};
%\item
%\end{itemize}
the initial condition  $\tilde{u}(x,T_1)=M_\ve u^\ve_0 (x).$ Of course, the restriction on $\widetilde{Q}_T$ of the solution $u$ of \eqref{PS}, 
%(that was defined on all $Q_T=\Omega\times(0,T))$ 
corresponding to the multiplicative control $v$, given in \eqref{piecew}, and to the initial state $u_0,$ solves \eqref{PSfinal} with the initial state $u(x,T_1)=M_\ve u^\ve_0 (x)+\d_0^\ve(x).$
% the 
% the solution $u(x,t) $ %of \eqref{PS} 
%%is still represented  as the sum of  $ w (x,t) $ and $h (x,t)$, which 
%solves the problem \eqref{PS} in $(-1,1)\times(T_1,T),$ with the state $S_\ve u^\ve_{0}$ at time $T_1,$ that is we consider the following problem. 
%Let us also denote with $u(x,t)$ the unique strict solution of the following problem
%
% $u(x,T_1)=S_\ve u^\ve_0 (x)$
Since the inequality \eqref{Sve} holds we can apply Theorem \ref{lem preserving} to steer the system \eqref{PSfinal} from $M_\ve u^\ve_{0}$ to the approximation $u^*_\eta,$ thus, for $T>T_1$ we have % $\widetilde{u}$
%Thus, from \eqref{uT} and \eqref{bound} there exists $T_2\in (T_1, T_1+1)$ such that for every $T\in (T_1, T_2)$ we have
  \begin{equation}\label{bound2}
  \!\!\!\!\!\!\|\widetilde{u} (\cdot,T)-%\overline{u}
u^*_\eta (\cdot)\|<\frac{\ve}{4}.
   \end{equation}
% {\bf Step 3:} {\it Conclusions.}  
 Then, using Corollary \ref{cor PDE2} (see also Proposition \ref{lem PDE}), from \eqref{bound2}, \eqref{approx u}, \eqref{5.1}, and \eqref{sigma}, for any $T$ such that $T-T_1>0$ is sufficiently small, there exists
%  $T^*\in (T_1,\min\{T_2,T_1+\frac{1}{4\nu}\})$ such that for every $T\in(T_1,T^*]$ we obtain 
% %and Since $\!\|u^*_\ve%(\cdot)
%%%-u^*%(\cdot)
%%\|_{L^2 (-1,1)}<\frac{\ve}{4},$ 
%from \eqref{5.12}, %, \eqref{approx u} 
%\eqref{5.10} and \eqref{5.1bis} we obtain
%\begin{multline*}%\label{C1data} 
%\|u(\cdot,T) -u^*(\cdot)\|\leq 
%\|u(\cdot,T)-\widetilde{u}(\cdot,T)%{\overline{\rho}}
%\|%_{L^2 (-1,1)}\!\!
%%+\!\|\overline{u}_{\overline{\rho}}%(\cdot)
%%-\overline{u}_\ve%(\cdot)
%%\|_{L^2 (-1,1)}\!\!
%+ \|\widetilde{u}(\cdot,T) -u^*_\ve(\cdot)\|+\|u^*_\ve-u^*\|\\
%\leq
%%e^{\nu (T-T_1)}
%\sqrt{2}\|S_\ve u^\ve_{0}+\sigma_0^\ve-S_\ve u^\ve_{0}\|+\frac{\ve}{6}+\frac{\ve}{6}<\frac{\ve}{2},
%\end{multline*}
%Thus, for any $T$ such that $T-T_1>0$ is sufficiently small, from \eqref{5.12}, \eqref{approx u}, \eqref{5.1} we obtain 
\begin{align*}
\|  u(\cdot,T) -u^*(\cdot) \|%_{L^2 (\Omega)}
&\leq  \|  u(\cdot,T) - u^*_\eta(\cdot) \|%_{L^2 (\Omega)}
+\|    u^*_\eta -u^*  \|%_{L^2 (\Omega)}
\\
&\leq \|u(\cdot,T)-\widetilde{u}(\cdot,T)\|+ \|  \widetilde{u}(\cdot,T) - u^*_\eta(\cdot) \|%_{L^2 (\Omega)}
+\|    u^*_\eta -u^*_\ve  \|%_{L^2 (\Omega)}
+\|    u^*_\ve -u^*  \|%_{L^2 (\Omega)}
\\
&< \sqrt{2}\|M_\ve u^\ve_{0}+\d_0^\ve-M_\ve u^\ve_{0}\|+\frac{3}{4}\ve<\ve\,.% \frac{\ve}{2}+
\end{align*}
 From which it follows the conclusion and  the approximate controllability at any time $T>0,$ using the same approach of the end of the proof of Theorem \ref{lem preserving}.
 %%given in \eqref{Vve}
%%the following function defined on $\overline{\Omega}$
%By \eqref{nu} %(see \eqref{ass}) 
%$ v_0 \in L^\infty (\Omega),$ and
%by \eqref{ass step} we have 
%$v_0 (x) \; \leq 0,\;\forall x\in \overline{\Omega}.$
%Let us select the bilinear control
%$$
%v (x,t) \; := \; \frac{1}{T} v_0 (x)\;\;\;\forall x\in\Omega\times(t_1,T).
%$$
%Proposition \ref{lem preserving}
%
%We remark that, thanks to \eqref{ass step}-\eqref{nu}, the assumption %\eqref{Linfinity} and 
%\eqref{ass} of  holds. Then proceeding similarly to Lemma \ref{lem preserving}, with a proof essentially identical, there  exists $T=T(\eta)>t_1$ with $T-t_1$ sufficiently small ($0<T-t_1\ll %\in(0,\frac{1}{2L})
%\frac{1}{4L}
%$)
%        we can obtain the following inequality (similar to \eqref{ineq pres} of Lemma \ref{lem preserving}) 
%\begin{equation}\label{5.12}
%\|u (\cdot, T) - \overline{u}_{\rho} \|_{L^2(\Omega)}\leq \frac{\eta}{4}+\sqrt{2}\|u (\cdot, t_1) - K u_{in} (\cdot)\|_{L^2(\Omega)},%+\vp (\cdot, T) + h (\cdot, T), % \rho^* (\cdot, T),  
%\end{equation}
\end{proof}

\section*{Acknowledgment}
This work was supported by the Istituto Nazionale di Alta Matematica (IN$\delta$AM),
through the GNAMPA-IN$\delta$AM Research Project 2020, by the title (in Italian language) \lq\lq Problemi inversi e di controllo per equazioni di evoluzione e loro applicazioni'', coordinator Giuseppe Floridia.
Moreover, this research was performed in the framework of the 
%GDRE CONEDP (European Research Group on \lq\lq Control of Partial Differential Equations'') issued by CNRS, INdAM and Universit\'e de Provence, then in the framework of the 
French-German-Italian Laboratoire International Associ\'e (LIA), named COPDESC, on Applied Analysis,  issued by CNRS, MPI and IN$\delta$AM.
%This paper was also supported by the research project of the University of Naples Federico II: \lq\lq Spectral and Geometrical Inequalities''.

%\section*{References}
%They are to be cited in the text in superscript
%after comma and period (e.g.~word,\cite{ACF})
%but before other punctuation marks like colons,
%(e.g.~word\cite{ACF}:) semi-colons and\break
%question marks. If it is mentioned in the text as part of a sentence,
%it should be of normal size, e.g.~see~\cite{ACF}.

\end{document}